\documentclass[12pt,reqno]{amsart}
\usepackage[T1]{fontenc}
\usepackage[utf8]{inputenc}
\usepackage{amssymb,latexsym,amsthm,amsfonts,amscd,pgf,enumerate,ragged2e}
\usepackage{amsmath}
\usepackage{pstricks-add}
\usepackage{graphics,pstricks}
\usepackage{pgf,tikz}
\usepackage{mathrsfs}
\usepackage{fullpage}
\usepackage{microtype,textcomp,csquotes, lmodern}

\usetikzlibrary{arrows}
\usepackage{hyperref}
\usepackage{cleveref}

\newtheorem{proposition}{Proposition}[section]
\newtheorem{lemma}[proposition]{Lemma}
\newtheorem{theorem}[proposition]{Theorem}
\newtheorem{corollary}[proposition]{Corollary}
\newtheorem{question}[proposition]{Question}
\newtheorem{conjecture}[proposition]{Conjecture}
\theoremstyle{definition}
\newtheorem{remark}[proposition]{Remark}

\newtheorem{example}[proposition]{Example}
\newtheorem{definition}[proposition]{Definition}

\newcommand{\Indr}[1]{{\rm Ind}_{#1}}

\tikzstyle{place}=[draw,circle,minimum size=1mm,inner sep=1pt,outer sep=-1.1pt,fill=black]

\usetikzlibrary{shapes.geometric}
\tikzstyle{places}=[draw,rectangle,minimum size=8pt,inner sep=0pt]
\tikzstyle{placesf}=[draw,rectangle,minimum size=5pt,inner sep=0pt]
\tikzstyle{placec}=[draw,circle,minimum size=8pt,inner sep=0pt]
\tikzstyle{placecf}=[draw,circle, minimum size=5pt,inner sep=0pt]

\begin{document}


\title{Fr\"oberg's Theorem, vertex splittability and higher independence complexes}

\author{Priyavrat Deshpande}
\address{Chennai Mathematical Institute, India}
\email{pdeshpande@cmi.ac.in}
\thanks{}

\author{Amit Roy}
\address{Chennai Mathematical Institute, India}
\email{amitiisermohali493@gmail.com}
\thanks{}

\author{Anurag Singh}
\address{Indian Institute of Technology Bhilai, India}
\email{asinghiitg@gmail.com}
\thanks{}

\author{Adam Van Tuyl}
\address{Department of Mathematics \& Statistics, McMaster University, Hamilton, ON L8S 4L8, Canada}
\email{vantuyl@math.mcmaster.ca}
\thanks{}

\date{March 09, 2024}


\begin{abstract}
A celebrated theorem of Fr\"oberg gives a complete combinatorial classification of quadratic square-free monomial ideals with a linear resolution. 
A generalization of this theorem to higher degree square-free monomial ideals is an active area of research. 
The existence of a linear resolution of such ideals often depends on the field over which the polynomial ring is defined. 
Hence, it is too much to expect that in the higher degree case a linear resolution can be identified purely using a combinatorial feature of an associated combinatorial structure. 
However, some classes of ideals having linear resolutions have been identified using combinatorial structures. 
In the present paper, we use the notion of $r$-independence to construct an $r$-uniform hypergraph from the given graph. 
We then show that when the underlying graph is co-chordal, the corresponding edge ideal is vertex splittable, a condition stronger than having a linear resolution. 
We use this result to explicitly compute graded Betti numbers for various graph classes. 
Finally, we give a different proof for the existence of a linear resolution using the topological notion of $r$-collapsibility. 

\end{abstract}

\maketitle


\section{Introduction}

Let $G$ be a finite simple graph with $V(G)=\{x_1,\ldots, x_n\}$ as its vertex set and $E(G)$ be its edge set. 
The independence complex of $G$, denoted by $\mathrm{Ind}(G)$, is the simplicial complex whose simplices are independent subsets of vertices in $G$. The
complex ${\rm Ind}(G)$ is an important object in combinatorics which lies at the crossroads of various fields of mathematics and computer science. 
For example, via {\it edge ideals}, a concept introduced by Villarreal \cite{Vill}, the independence complex appears in commutative algebra. 
In particular, let $R=\mathbb K[x_1,\ldots,x_n]$ be the polynomial ring in $n$ variables over a field $\mathbb K$. 
Then the edge ideal $I(G)$ of $G$ is the quadratic square-free monomial ideal $\langle x_ix_j\mid \{x_i,x_j\}\in E(G)\rangle$ of $R$. 
The ideal $I(G)$ is also the {\it Stanley-Reisner ideal} of $\mathrm{Ind}(G)$. 
Determining algebraic and homological properties of the ideal $I(G)$ in terms of the combinatorial properties of $\mathrm{Ind}(G)$ is an active area of research in commutative algebra. 

In 2018 Paolini and Salvetti \cite{PS} considered a generalisation of the independence complex in the context of braid groups, called the {\it $r$-independence complex} $\mathrm{Ind}_r(G)$ of $G$, for any positive integer $r$. 
A subset $A\subseteq V(G)$ is called $r$-independent if each connected component of the induced subgraph $G[A]$ has at most $r$ vertices. 
The collection of all $r$-independent sets forms the simplicial complex $\mathrm{Ind}_r(G)$. 
Note that $\mathrm{Ind}_1(G)$ is the independence complex of $G$.

The main focus of Paolini and Salvetti \cite{PS} was to understand twisted (co)homology groups of the classical braid groups via relating them to that of  $r$-independence complexes of certain graphs. 
Their results indicated that these complexes are interesting in their own right.
Later, it was proved in \cite{DS} that the $r$-independence complexes of cycle graphs and perfect $m$-ary trees are homotopy equivalent to a wedge of spheres. 
Extending a result of Meshulam \cite{Meshulam}, it was also proved in \cite{DSS} that the (homological) connectivity of $r$-independence complexes of graphs gives an upper bound for the distance $r$-domination number of graphs. 
In the same paper the authors proved that the
$r$-independence complexes of chordal graphs are homotopy equivalent to a wedge of spheres for all $r\ge 1$. 
From the perspective of the Cohen-Macaulay property, it was shown in \cite{FMA} that the $r$-independence complexes of trees are shellable. 
Moreover, it was also proved, using commutative algebra techniques, that the $r$-independence complexes of caterpillar graphs are vertex decomposable, 
a property that implies the Cohen-Macaulay property.

In this article we focus on developing the algebraic properties of the $r$-independence complexes. 
In particular, we focus on various algebraic and homological invariants of the Stanley-Reisner ideal $I_r(G)$ of $\mathrm{Ind}_r(G)$. 
The ideal $I_r(G)$ can also be viewed as the hyperedge ideal of a hypergraph associated to certain subgraphs of $G$.  Specifically, 
let $\mathrm{Con}_r(G)$ denote the hypergraph with the vertex set $V(G)$ and the hyperedges are those $r+1$-subsets $W$ such that the induced subgraph $G[W]$ is connected. 
Then the Stanley-Reisner ideal of $\mathrm{Ind}_r(G)$ is same as the edge ideal of $\mathrm{Con}_r(G)$. 
Note that $\mathrm{Con}_1(G)=G$. Thus $I_1(G)$ is nothing but the usual {\it edge ideal} of $G$. 
Moreover, if $G=K_n$, the complete graph, or if $G=K_{n_1,\ldots,n_t}$, the complete multipartite graph, then $\mathrm{Con}_r(G)$ matches with the complete hypergraph $K_n^{r+1}$ and a $(r+1)$-complete multipartite hypergraph $K_{n_1,\ldots,n_t}^{r+1}$ defined by Emtander in \cite{Emtander}.  
 
In 1988 Fr\"{o}berg \cite{RF} showed that the complement of a graph is chordal if and only if the Stanley-Reisner ideal of its independence complex has a linear free resolution. 
We are interested in determining whether or not the Stanley-Reisner ideal of the $r$-independence complex of a graph has a linear free resolution, given that the complement of the graph is chordal. 
We show that this is indeed true, thus giving a partial generalization 
of Fr\"oberg's result:

\begin{theorem}\label{maintheorem}
    If $G$ is a graph whose complement is
    chordal, then $I_r(G)$ has a $(r+1)$-linear resolution.
\end{theorem}

\noindent
In fact, we give two different proofs. One proof is by using the notion of vertex splittable ideals from commutative algebra (see Corollary \ref{cor.froberg1}) and another one is by showing that the complex is $r$-collapsible, which is a key concept in topological combinatorics (see Theorem \ref{theorem:chordal collapsing}).  Note that the converse of Theorem \ref{maintheorem} does
not hold for $r \geq 2$; for more, see Corollary \ref{cor.conversefalse}.

One of the most useful ways to study the structure of a module is by analyzing the minimal free resolution of the module. 
Important numerical invariants of the free resolution are its graded Betti numbers. 
Using the fact that  vertex splittable ideals admit a {recursive formula for
the graded Betti numbers}, we are able to provide explicit formulas for the $\mathbb N$-graded Betti numbers of { $I_r(G)$ for various
families of graphs.}

This paper is organised as follows.  In \Cref{preliminaries} we recall the
relevant graph theory and Stanley-Reisner theory.  In \Cref{vertex spllitable}, we
use the notion of vertex splittable ideals to give our first proof
of Theorem \ref{maintheorem}. In \Cref{Betti numbers}  we deduce results
about the graded Betti numbers of $I_r(G)$ for some families of graphs
$G$. 
In \Cref{collapsibility} we provide an alternative proof to Theorem \ref{maintheorem}
that uses tools from topological combinatorics.
Finally, in \Cref{conrem} we outline some questions for future research. 

\medskip 
\noindent
{\bf Acknowledgements.}  
The authors would like to thank the referee for a quick and careful reading and for some helpful suggestions; especially for suggesting to use the Hilbert series for computations of Betti numbers in \Cref{Betti numbers} of the paper. Priyavrat Deshpande and Amit Roy are partially supported by a grant from the Infosys Foundation. 
Anurag Singh is partially supported by the Start-up Research Grant SRG/2022/000314 from SERB, DST, India. Van Tuyl’s research is partially
supported by NSERC Discovery Grant 2019-05412. 
 

\section{Stanley-Reisner ideals of higher independence complexes}\label{preliminaries}

In this section we recall some relevant results
from graph theory and Stanley-Reisner theory.  
In particular, we state some basic properties of Stanley-Reisner ideals of higher independence complexes.  
For more on Stanley-Reisner theory and monomial ideals, see \cite{HH,VBook}.

\subsection{Graph terminology} 
Throughout this paper, $G = (V(G),E(G))$ denotes a finite
simple graph on the vertex set $V(G)= \{x_1,\ldots,x_n\}$
and the set $E(G)$ of edges which is a collection of $2$-element subsets of $V(G)$.
If $x$ is a vertex of $G$, then $|\{y\in V(G)\mid \{x,y\}\in E(G)\}|$ is called the {\it degree} of $x$ in $G$, and is denoted by $\deg (x)$. If $\deg(x)=1$, then $x$ is called a leaf of $G$.
For $x\in V(G)$, $G\setminus x$ denotes the graph with vertex set $V(G)\setminus\{x\}$ and edge set $\{\{u,v\}\in E(G)\mid x\notin\{u,v\}\}$. 
The complement $G^c$ of $G$ is a graph with vertex set $V(G^c)=V(G)$ and $E(G^c)=\{\{x,y\}\mid \{x,y\}\notin E(G)\}$. 
The {\it neighbourhood} of $x$ in $G$ is defined as $N_G(x):=\{y\in V(G)\mid \{x,y\}\in E(G)\}$. 
The set $N_G(x)\cup\{x\}$ is called the {\it closed neighbourhood} of $x$ in $G$, and is denoted by $N_G[x]$. 
A {\it connected component} of $G$ is a (maximal) subgraph of $G$ such that for every pair of vertices in the subgraph,
there is a path within the subgraph that connects these
vertices. 
If $G$ is disconnected and $\mathcal{C}_1,\mathcal C_2,\ldots,\mathcal C_k$ are the connected components of $G$, then for each $i$, $V(\mathcal C_i)$ denotes the vertex set of the connected component $\mathcal C_i$.

If $G = (V(G),E(G))$ is a graph and $A \subseteq V(G)$, then
the {\it induced subgraph of $G$ on $A$}, denoted
$G[A]$, is the graph with vertex set $V(G[A]) = A$ and edge
set $E(G[A]) = \{e \in E(G) ~|~ e \subseteq A\}$. 
Let $A=\{x_{i_1},\ldots,x_{i_k}\}\subseteq V(G)$ be such that $G[A]\cong C_k$, a cycle of length $k$. 
If $E(G[A])=\{\{x_{i_j},x_{i_{j+1}}\}\mid 1\le j\le k-1\}\cup\{x_{i_k},x_{i_1}\}$, then we simply write the cycle $G[A]$ as $x_{i_1}\cdots x_{i_k}$.

A subset $W \subseteq V(G)$ is called an {\it $r$-independent set} of $G$ if each connected component of $G[W]$ has at most $r$ vertices. Note that a $1$-independent set is the usual {\it independent set} in a graph $G$.
We call an $r$-independent set $W$ a {\it maximal 
$r$-independent set} if $W$ is maximal with respect to inclusion among all
$r$-independent sets.


\subsection{Simplicial Complexes}
Fix a set of vertices $V = \{x_1,\ldots,x_n\}$.  A
{\it simplicial complex} $\Delta$ on $V$ is a subset
of $2^V$, that satisfies the
properties that $\{x_i\} \in \Delta$ for all
$i=1,\ldots,n$ and if $F \in \Delta$ and $G \subseteq
F$, then $G \in \Delta$. If $\Delta=2^V$, then $\Delta$ is called a {\it simplex} (or
$(n-1)$-simplex if we want to highlight
the number of vertices). An element $F \in \Delta$ is called a {\it face} of the simplicial complex $\Delta$.  
A face in $\Delta$ that is maximal with respect
to inclusion is called a {\it facet}.  If
$\{F_1,\ldots,F_s\}$ is a complete list of the
facets of $\Delta$, then we sometimes write
$\Delta = \langle F_1,\ldots,F_s \rangle$ and say that $\Delta$ is generated by $F_1,\ldots,F_s$.  The
{\it dimension} of a face $F$ is $\dim F = |F|-1$,
while the {\it dimension} of a simplicial
complex $\Delta$ is $\dim \Delta = \max\{ \dim F ~|~ F \in \Delta\}$. A simplicial complex is said to be {\it pure} if $\dim F = \dim \Delta$ for all facets of $\Delta$. If a simplex has dimension $d$, then it is called a $d$-simplex. For $k\le d$, the $k$-skeleton of the $d$-simplex $\Delta$ on the vertex set $Y=\{x_1,\ldots,x_{d+1}\}$ is the collection of all subsets of $Y$ which have cardinality at most $k+1$.

Given a simplicial complex $\Delta$ on the vertex set $V=\{x_1,\ldots,x_n\}$ we can associate it with a square-free
monomial ideal $I_{\Delta}$ in the polynomial ring 
$R =\mathbb K[x_1,\ldots,x_n]$ over the field $\mathbb K$ in the following way. For $A \subseteq V$, we write 
$$\mathbf x_A = \prod_{x_i \in A} x_i,$$
to denote the monomial obtain by multiplying together all the variables corresponding to the vertices in $A$. Then the ideal 
\[
I_\Delta = \langle \mathbf x_A ~|~ \mbox{$A \subseteq V$ and
$A \not\in \Delta$} \rangle
\]
is called the Stanley-Reisner ideal of $\Delta$.
The ideal $I_\Delta$ and the ring $R/I_\Delta$ (sometimes
called the Stanley-Reisner ring) captures invariants
of the simplicial complex.  For example, the 
Krull dimension of $R/I_\Delta$ satisfies
\begin{equation}\label{dim}
\mbox{K-dim}(R/I_\Delta) = \dim \Delta +1
\end{equation}
Expanding this dictionary between the algebraic
invariants of $R/I_\Delta$ and $\Delta$ when
$\Delta$ is a simplicial complex constructed from a
graph is the focus for the remainder of the paper.

\subsection{Stanley-Reisner ideals of higher independence complexes}

We now formally introduce the higher independence complexes of a graph $G$, the main object of study in this paper.

\begin{definition}\label{defn.rindcplx}
Let $G = (V(G),E(G))$ be a finite simple graph with
$V(G) = \{x_1,\ldots,x_n\}$, and let $r$ be any positive integer. Then the
{\it $r$-independence complex} of $G$,
denoted by ${\rm Ind}_r(G)$,
is the simplicial complex 
\[
{\rm Ind}_r(G) = \{ W \subseteq V(G) ~|~
\mbox{$W$ is an $r$-independent set} \}
\]
on the vertex set $V(G)$.
\end{definition}

When $r=1$, ${\rm Ind}_1(G)$ is the independence complex of $G$. 
If $r\ge n$, then ${\rm Ind}_r(G) = \langle \{x_1,x_2,\ldots,x_n\} \rangle$ since $V(G)$ is the $r$-independent set. 
For $r=n-1$, if $G$ is disconnected then $\mathrm{Ind}_{n-1}(G)$ is the $(n-1)$-simplex $\langle \{x_1,x_2,\ldots,x_n\} \rangle$ and if $G$ is connected then $\mathrm{Ind}_{n-1}(G)$ is the $(n-2)$-skeleton of the simplex $\langle \{x_1,x_2,\ldots,x_n\}\rangle$.

We now describe the Stanley-Reisner ideal of ${\rm Ind}_r(G)$.

\begin{theorem}
Let $G = (V(G),E(G))$ be a finite simple graph with
$V(G) = \{x_1,\ldots,x_n\}$.
Then the Stanley-Reisner ideal of ${\rm Ind}_r(G)$ 
in $R = \mathbb{K}[x_1,\ldots,x_n]$ is
$$I_{{\rm Ind}_r(G)} = \langle \mathbf x_A ~|~ \mbox{$A \subseteq V(G)$ with
$|A|=r+1$ and $G[A]$ connected}~\rangle.$$
\end{theorem}

\begin{proof}
Let $J=\langle \mathbf x_A ~|~ A \subseteq V(G)\text{ with }
|A|=r+1 \text{ and } G[A] \text{ connected }\rangle$. Suppose that $A \subseteq V$ with $|A|=r+1$ and $G[A]$ connected.
Then $A$ is not an $r$-independent set, i.e., $A \not\in {\rm Ind}_r(G)$, and consequently, $\mathbf x_A$ is in $I_{{\rm Ind}_r(G)}$.  
So $J \subseteq I_{{\rm Ind}_r(G)}$ since all of its generators
are in $I_{{\rm Ind}_r(G)}$.

Consider any square-free monomial $x_{i_1}\cdots x_{i_t} \in I_{{\rm Ind}_r(G)}$.
So $W = \{x_{i_1},\ldots,x_{i_t}\} \subseteq V(G)$ is not
in ${\rm Ind}_r(G)$.  This means that $G[W]$ has a connected 
component with at least $r+1$ vertices.  Thus, we can find a subset $A \subseteq W$ with $|A|=r+1$ such that
$G[A]$ is connected.  So $\mathbf x_A \in J$, and $\mathbf x_A$ divides 
$x_{i_1}\cdots x_{i_t}$.  Since both $I_{\mathrm{Ind}_r(G)}$ and $J$ are square-free monomial ideals we have $I_{{\rm Ind}_r(G)} \subseteq J$.
\end{proof}

\begin{remark}
The ideal $I_{\mathrm{Ind}_r(G)}$ is sometimes denoted by $I_r(G)$. When $r=1$ in the previous Theorem, $I_1(G)=I_{{\rm Ind}_1(G)} = \langle x_ix_j ~|~ \{x_i,x_j\} \in E(G) \rangle$
is the well-known {\it edge ideal} $I(G)$ of $G$. Also observe that 
when $G$ is a connected graph on $n$ vertices, then
$I_{n-1}(G) = \langle x_1x_2\cdots x_n \rangle$. If $G$ is disconnected and $|V(G)|=n$, then $I_r(G)=\langle 0\rangle$, for $r\ge n-1$.
\end{remark}

When $r=1$, i.e., when working with edge ideals, specifying two adjacent vertices is equivalent to specifying an edge.
However, when considering $r>1$, the ideal may not necessarily \emph{specify} the induced subgraphs on $r+1$ vertices; it merely specifies whether or not they are connected. 
It is well known that graphs and square-free monomial ideals generated in degree $2$ are in one-to-one correspondence. 
However, for $r>1$, non-isomorphic graphs may have the same $r$-independence complex (and hence have the same Stanley-Reisner ideals). 
A trivial example is that any two connected graphs on $n$ vertices will have the same $(n-1)$-independence complex. 
A less trivial example is given below.

\begin{example}
    Let $G$ be the complete graph $K_5$ on five vertices $X=\{x_1,\ldots,x_5\}$. 
    Let $r>1$ and $H$ be the graph obtained from $K_5$ by removing the edge $\{x_1,x_2\}$. 
    Note that the facets of $\mathrm{Ind}_r(G)$ are all $F\subseteq\{x_1,\ldots,x_5\}$ such that $|F|=r$. 
    Now suppose $F'\subseteq X$ is a facet of $\mathrm{Ind}_r(H)$. 
    By definition $|F'|\ge r$. If $\{x_1,x_2\}\nsubseteq F'$, then for each $x_i,x_j\in F'$, $\{x_i,x_j\}\in E(H)$. 
    Hence $|F'|\le r$. If $\{x_1,x_2\}\subseteq F'$, then for any $i\notin \{1,2\}$, $H[x_1,x_2,x_i]$ is a connected subgraph of $H$. 
    Thus $|F'|\le r$. Consequently, $|F'|=r$. 
    Hence $\mathrm{Ind}_r(G)=\mathrm{Ind}_r(H)$. The graphs $G$ and $H$ are non-isomorphic graphs since every vertex in $G$ has degree four whereas the vertex $x_1$ in $H$ has degree three. 
\end{example}

The previous example is just a special
case of a more general phenomena.

\begin{proposition}
    Let $K_n$ be the complete graph on $n$ vertices. Suppose $V(K_n)=\sqcup_{i=1}^k W_i$ with $\max_{i=1}^k\{|W_i|\} =s$. If $H$ is the graph obtained from $K_n$ by deleting some edges (maybe all or none) from each $K_n[W_i]$, then for all $r\ge s$, $\mathrm{Ind}_r(K_n)=\mathrm{Ind}_r(H)$.
\end{proposition}

\begin{proof}
The statement is true for all $r \geq n$ by
the discussion after Definition \ref{defn.rindcplx}. So, we can assume $s \leq r < n$.
It is enough to show that if $F\subseteq V(K_n)$ is such that $s \leq |F|=r$, then $F$ is a facet of $\mathrm{Ind}_r(H)$ and these are all the facets of $\mathrm{Ind}_r(H)$. Now suppose $F$ is a facet of $\mathrm{Ind}_r(H)$, then $|F|\ge r$. If $W_i\nsubseteq F$ for all $i$, then the induced subgraph $H[F]$ is connected and hence $|F|\le r$. Now suppose $W_i\subseteq F$ for some $i$. In that case for all $x\in F$, with $x\notin W_i$, $H[W_i\cup\{x\}]$ is a connected subgraph of $H$. Note that $|W_i|\le r$ since $\max_{i=1}^k\{|W_i|\}= s \le r$. Therefore, $F=W_i$, if $|W_i|=r$ and $H[F]$ is connected if $r> |W_i|$. Consequently, $|F|=r$ and hence $\mathrm{Ind}_r(K_n)=\mathrm{Ind}_r(H)$.
\end{proof}


\section{Fr\"oberg's theorem via vertex splittable ideals}\label{vertex spllitable}

In this section, we show that one direction of Fr\"oberg's theorem extends very naturally to the Stanley-Reisner ideals of higher independence complexes. 
On the other hand, the natural converse of this theorem does not hold. 

We first recall the definition of a (linear) resolution.
Given any homogeneous ideal $I \subseteq R = \mathbb{K}[x_1,\ldots,x_n]$, the {\it graded minimal free resolution} of $I$ is the long exact sequence 
$$0\rightarrow \bigoplus_{j \in \mathbb{N}} R(-j)^{\beta_{p,j}(I)}
\rightarrow \bigoplus_{j \in \mathbb{N}} R(-j)^{\beta_{p-1,j}(I)}
\rightarrow \cdots \rightarrow \bigoplus_{j \in \mathbb{N}}
R(-j)^{\beta_{0,j}(I)} \rightarrow I \rightarrow 0,$$
where $R(-j)$ denotes the polynomial ring with the grading
twisted by $j$ and $p \leq n$.  The numbers 
$\beta_{i,j}(I)$ are called the {\it $(i,j)$-th graded Betti
numbers of $I$}.  See \cite{VBook} for more on the graded
resolution.

\begin{definition}
Let $I \subseteq R$ be a homogeneous ideal, and suppose
that all the generators of $I$ have degree $d$.
Then  $I$ has a {\it linear resolution} if 
$\beta_{i,j}(I) =0$ for all $j \neq i+d$.
\end{definition}

\begin{remark}
    The {\it regularity} of an ideal $I$ is defined
    to be ${\rm reg}(I) = \max\{j-i ~|~ \beta_{i,j}(I) \neq 0\}$.
    If $I$ is a homogeneous ideal generated in degree $d$,
    then $I$ has a linear resolution if and only if 
    ${\rm reg}(I)=d$.
\end{remark}

\begin{definition}
A simple graph $G$ is called \emph{chordal} if there are no induced cycles of length four or more.      
The complement of a chordal graph is called \emph{co-chordal}.
\end{definition}

Fr\"oberg \cite{RF} classified which quadratic square-free
monomial ideals have a linear resolution.

\begin{theorem}[Fr\"oberg's Theorem]  Let $G$ be a finite
simple graph.  Then  the edge ideal $I(G) = I_{{\rm Ind}_1(G)}$ has a linear resolution (equivalently regularity is $2$) if and only if $G$ is co-chordal.
\end{theorem}

We will prove the ``if'' direction of Fröberg's theorem in the context of higher independence complexes; i.e., the Stanley-Reisner ideal of $\Indr{r}(G)$ has a linear free resolution whenever $G$ is the complement of a chordal graph. 

Let $I$ be a square-free monomial ideal in a polynomial ring $R=\mathbb K[x_1,\ldots,x_n]$ and let $\mathcal{G}(I)$ denote the unique set of minimal generators of the ideal $I$. 
We require the notion of a vertex splittable ideal,
first introduced by Moradi and Khosh-Ahang \cite{SF1}.


\begin{definition}\label{vertex splittable ideal}
\normalfont
 We say that a monomial ideal $I$ is {\it vertex splittable} if $I$ can be obtained by the following recursive procedure:
 \begin{enumerate}
  \item[(i)] If either $I=\langle m\rangle$ where $m$ is a monomial, or $I=\langle 0\rangle$, or $I=R$.
  
  \item[(ii)] If there exists a variable $x_i$ and two vertex splittable ideals $I_1$ and $I_2$ of the polynomial ring $\mathbb K[x_1,\ldots,\widehat{x_i},\ldots,x_n]$ such that $I=x_iI_1+I_2$ with $I_2\subseteq I_1$ and the minimal generators of $I$ is the disjoint union of the minimal generators of $x_iI_1$ and $I_2$. 
 \end{enumerate}
\end{definition}

\begin{lemma}\label{variable}
    Let $I$ be an ideal of $R$ such that $I$ is generated by variables. Then $I$ is vertex splittable.
\end{lemma}
\begin{proof}
    This follows by applying induction on the number of generators of $I$. 
\end{proof}
\begin{lemma}\label{extra variable product}
    Let $I$ be a vertex splittable ideal of $R$. Then the ideal $x_{n+1}I\subseteq R[x_{n+1}]$ is also vertex splittable.
\end{lemma}
\begin{proof}
    We prove this by induction on $n$. If $n=1$, then $I=\langle 0\rangle$, or $\langle 1\rangle$, or $\langle x_1\rangle$. Hence $x_2I$ is a vertex splittable ideal.

    Let us assume that $n\ge 2$. If $I=\langle 0\rangle$, or $\langle 1\rangle$, or $\langle m\rangle$, where $m$ is a monomial, then we can see that  $x_{n+1}I$ is a vertex splittable ideal. Now, suppose that
    \[
    I=x_iI_1+I_2,
    \]
    \noindent
    where $I_1$ and $I_2$ are vertex splittable ideals of $\mathbb K[x_1,\ldots,\widehat{x_i},\ldots, x_n]$ with $\mathcal G(I)=\mathcal G(x_iI_1)\sqcup\mathcal G(I_2)$ and $I_2\subseteq I_1$. Then
    \[
    x_{n+1}I=x_i(x_{n+1}I_1)+x_{n+1}I_2.
    \]
    \noindent
    By induction $x_{n+1}I_1$ and $x_{n+1}I_2$ are vertex splittable ideals of $\mathbb K[x_1,\ldots,\widehat{x_i},\ldots, x_n,x_{n+1}]$ with $\mathcal G(x_{n+1}I)=\mathcal G(x_ix_{n+1}I_1)\sqcup\mathcal G(x_{n+1}I_2)$ and $x_{n+1}I_2\subseteq x_{n+1}I_1$. Therefore, $x_{n+1} I$ is a vertex splittable ideal.
    \end{proof}
There is another characterization of chordal graphs that we will use frequently. 
A vertex $v$ of a graph is called a \emph{simplicial vertex} if the induced subgraph on $v$ and its neighbours 
$N_G(v)$ is a clique. 
A graph $G$ is chordal if and only if there is a subset of vertices $\{v_1,\dots, v_n\}$ such that $v_{i}$ is simplicial in the graph induced on $V(G)\setminus\{v_1,\dots,v_{i-1}\}$ for $i=1,\ldots,n$.

Moradi and Khosh-Ahang in \cite{SF1} showed that the edge ideal of a co-chordal graph $G$, i.e., the Stanley-Reisner ideal of the $1$-independence complex of $G$ is a vertex splittable ideal.

\begin{theorem}\cite[Theorem 3.6]{SF1}
    If $G$ is a co-chordal graph on the vertex set $V(G)$, then $I_1(G)=I_{\mathrm{Ind}_1(G)}$ is a vertex splittable ideal of $R$.
\label{r=1 case}
\end{theorem}

Our goal is to extend the above theorem to $I_{\mathrm{Ind}_r(G)}$ for $r>1$. 
In order to do this, we construct a new graph $\widetilde G$ from the given co-chordal graph $G$ 
with $x_1$ the simplical vertex
of $G^c$ as follows:
\begin{align*}
    V(\widetilde G)&=V(G)\setminus\{x_1\};\\
    E(\widetilde G)&=E(G\setminus x_1)\cup \{\{x_i,x_j\}\mid \{x_i,x_j\}\in E(G^c\setminus x_1)\text{ and } N_{G^c}(x_1)\cap\{x_i,x_j\}=\emptyset\}.
\end{align*}
Note that if $N_{G^c}(x_1)=Y$, then the induced subgraph of $G^c$ on the vertex set $\{x_1\}\cup Y$ is a clique since $x_1$ is a simplicial vertex of $G^c$. 
Thus we can rename the vertices of $G$ as $V(G)=V(G^c)=\{x_1\}\cup Y\cup W$ such that $\{x_1\}\cup Y$ is an independent set in $G$ and, $x_1$ is connected by an edge in $G$ to all the vertices in $W$. 
Informally, to construct $\widetilde G$ from $G$ we first remove the vertex $x_1$ and all its adjacent edges. 
Then we add  edges to the remaining vertices of $G$ according to the following rules. 
Among the vertices in $Y\sqcup W$, if $\{w_i,w_j\}\notin E(G\setminus x_1)$ for some $w_i,w_j\in W$, then we add it to $G$ since the induced subgraph on $\{x_1,w_i,w_j\}$ is connected in $G$. 
However, if $\{y_i,w_j\}\notin E(G\setminus x_1)$ for some $y_i\in Y$ and $w_j\in W$, then we don't add it to $\widetilde G$ since the induced subgraph on $\{x_1,y_i,w_j\}$ is disconnected in $G$. Moreover, $\{y_i,y_j\}\notin E(G\setminus x_1)$ for each $y_i,y_j\in Y$, and we also don't add it to $\widetilde G$ since $\{x_1,y_i,y_j\}$ is an independent set in $G$. Thus in $\widetilde G$, the vertices in $Y$ form an independent set and the induced subgraph on $W$ is a clique. In other words, $\widetilde{G}^c$ is the graph on the vertex set $Y\sqcup W$ such that the induced graph on $Y$ forms a clique, whereas the set $W$ forms an 
independent set.

We illustrate the construction of
$\widetilde G$ using an example.

\begin{example}\label{example 1}
\begin{figure}[!ht]
\centering
\begin{tikzpicture}
[scale=.55]
\draw [fill] (0.7,2) circle [radius=0.1];
\draw [fill] (1.7,1) circle [radius=0.1];
\draw [fill] (3.2,1) circle [radius=0.1];
\draw [fill] (1.7,3) circle [radius=0.1];
\draw [fill] (3.2,3) circle [radius=0.1];
\draw [fill] (4.7,3) circle [radius=0.1];
\node at (0,2) {$x_1$};
\node at (1.7,3.5) {$w_1$};
\node at (3.2,3.5) {$w_3$};
\node at (4.7,3.5) {$w_2$};
\node at (1.7,0.5) {$y_1$};
\node at (3.2,0.5) {$y_2$};
\node at (2.3,-0.5) {$G$};

\draw (3.2,3)--(0.7,2)--(1.7,3)--(1.7,1)--(4.7,3)--(0.7,2);
\draw (1.7,3)--(3.2,3)--(1.7,1);
\draw (3.2,3)--(3.2,1)--(4.7,3);

\draw [fill] (10.7,1) circle [radius=0.1];
\draw [fill] (12.2,1) circle [radius=0.1];
\draw [fill] (13.7,1) circle [radius=0.1];
\draw [fill] (15.2,1) circle [radius=0.1];
\draw [fill] (16.7,1) circle [radius=0.1];
\draw [fill] (12.2,2.5) circle [radius=0.1];

\draw (12.2,1)--(12.2,2.5)--(10.7,1)--(12.2,1)--(13.7,1)--(15.2,1)--(16.7,1);

\node at (10.7,0.5) {$x_1$};
\node at (12.2,3) {$y_1$};
\node at (12.2,0.5) {$y_2$};
\node at (13.7,0.5) {$w_1$};
\node at (15.2,0.5) {$w_2$};
\node at (16.7,0.5) {$w_3$};
\node at (14,-0.5) {$G^c$};
\end{tikzpicture}\caption{A co-chordal graph and its complement.}\label{figure 1}
\end{figure}
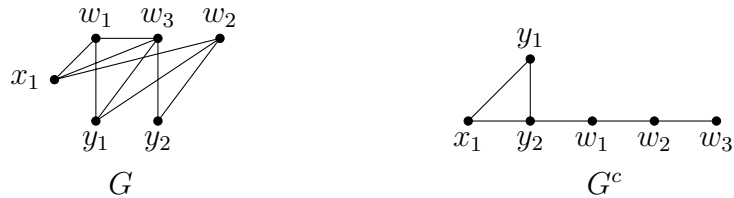

Let $G$ be the co-chordal graph in Figure \ref{figure 1}. Then $V(G)=\{x_1\}\cup Y\cup W$, where $Y=\{y_1,y_2\}$ and $W=\{w_1,w_2,w_3\}$. We see that $\{x_1,y_1,y_2\}$ forms an independent set in $G$ and $\{x_1,w_i\}\in E(G)$ for $1\le i\le 3$. The graph $\widetilde G$ is constructed from $G$ by first removing the vertex $x_1$ and its adjacent edges and then adding the edges $\{w_1,w_2\}$ and $\{w_2,w_3\}$.  Thus in $\widetilde G^c$, $W=\{w_1,w_2,w_3\}$ forms an independent set and the induced subgraph on  $Y=\{y_1,y_2\}$ is a clique. 

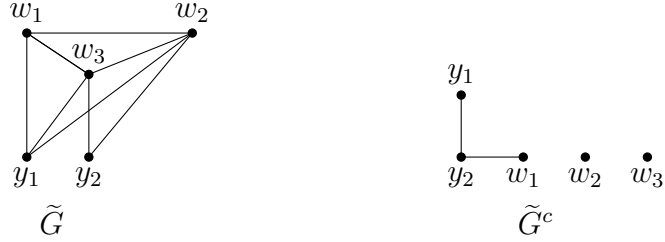
\begin{figure}[!ht]
\centering
\begin{tikzpicture}
[scale=.55]
\draw [fill] (1.7,1) circle [radius=0.1];
\draw [fill] (3.2,1) circle [radius=0.1];
\draw [fill] (1.7,4) circle [radius=0.1];
\draw [fill] (3.2,3) circle [radius=0.1];
\draw [fill] (5.7,4) circle [radius=0.1];

\node at (1.7,4.5) {$w_1$};
\node at (3.2,3.5) {$w_3$};
\node at (5.7,4.5) {$w_2$};
\node at (1.7,0.5) {$y_1$};
\node at (3.2,0.5) {$y_2$};
\node at (2.3,-0.5) {$\widetilde G$};

\draw (3.2,3)--(1.7,4)--(1.7,1)--(5.7,4);
\draw (5.7,4)--(1.7,4)--(3.2,3)--(1.7,1);
\draw (3.2,3)--(3.2,1)--(5.7,4)--(3.2,3);

\draw [fill] (12.2,1) circle [radius=0.1];
\draw [fill] (13.7,1) circle [radius=0.1];
\draw [fill] (15.2,1) circle [radius=0.1];
\draw [fill] (16.7,1) circle [radius=0.1];
\draw [fill] (12.2,2.5) circle [radius=0.1];

\draw (12.2,2.5)--(12.2,1)--(13.7,1);

\node at (12.2,3) {$y_1$};
\node at (12.2,0.5) {$y_2$};
\node at (13.7,0.5) {$w_1$};
\node at (15.2,0.5) {$w_2$};
\node at (16.7,0.5) {$w_3$};
\node at (14,-0.5) {$\widetilde G^c$};

\end{tikzpicture}\caption{The graph $\widetilde G$ and its complement.}\label{figure 2}
\end{figure}

\end{example}

Note that the graph $\widetilde G$ in the example is again co-chordal. 
It is in fact true that for any co-chordal graph $G$, the graph $\widetilde G$ is co-chordal, which we now prove.

\begin{proposition}\label{proposition 1}
    Let $G$ be a graph such that $V(G)=W\sqcup Y$, where $W$ forms an independent set in $G$, and the induced subgraph $G[Y]$ is a clique. Then both $G$ and $G^c$ are chordal.
\end{proposition}

\begin{proof}
    Note that, in $G^c$, $Y$ forms an independent set and $G^c[W]$ is a clique. Thus it is enough to show that $G$ is a chordal graph. Suppose $z_1\cdots z_p$ is a minimal cycle in $G$ of length at least four. Let $Z=\{z_1,\ldots,z_p\}$. If $|Z\cap Y|\ge 3$, then for any $y_{i_1},y_{i_2},y_{i_3}\in Z\cap Y$ the induced subgraph $G[\{y_{i_1},y_{i_2},y_{i_3}\}]$ would be a smaller cycle, a contradiction. Thus $|Z\cap Y|\le 2$. In that case, as $W$ forms an independent set in $G$ we must have $p\le 4$. Consequently, $p=4$. If $|Z\cap Y|=2$, then there exists some $y_{l_1},y_{l_2}\in Z\cap Y$ and $w_{m_1},w_{m_2}\in Z\cap W$ such that $G[Z]$ is the cycle $y_{l_1}w_{m_1}y_{l_2}w_{m_2}$. Note here that $y_{l_2}w_{m_2}y_{l_1}$ is  a smaller cycle in $G[Z]$, a contradiction. If $|Z\cap Y|\le 1$, then $|Z\cap W|\ge 3$. Hence there exists $w_{j_1},w_{j_2}\in Z\cap W$ such that $\{w_{j_1},w_{j_2}\}\in E(G)$, again a contradiction. Thus $G$ contains induced cycles of length at most $3$. This completes the proof.
\end{proof}

\begin{lemma}\label{aux lemma 1}
    Given a co-chordal graph $G$, let $\widetilde G$ be the graph constructed above. 
    Then $\widetilde G$ is also a co-chordal graph.
\end{lemma}

\begin{proof}
    In $\widetilde G^c$, we have $V(\widetilde G^c)=Y\sqcup W$ such that the vertices in $W$ form an independent set and the induced subgraph on $Y$ is a clique. Therefore, by Proposition \ref{proposition 1}, $\widetilde G^c$ is a chordal graph.
\end{proof}

Now we are ready to prove the main theorem of this section.

\begin{theorem}\label{general r case}
    If $G$ is co-chordal, then the ideal $I_r(G)=I_{\mathrm{Ind}_r(G)}$ is vertex splittable for all $r\geq 1$.
\end{theorem}
\begin{proof}
    We prove this by induction on $n$ and $r$. For a fixed $r$, if $n\in\{1,2,\ldots,r\}$ then $I_r(G)=\langle 0\rangle$ and hence a vertex splittable ideal. If $n=r+1$, then $I_r(G)=\langle 0\rangle$ if $G$ is a disconnected graph. Otherwise, $I_r(G)=\langle x_1\cdots x_{r+1}\rangle$. In both cases $I_r(G)$ is a vertex splittable ideal. Also, for $n \geq 1$, the case $r=1$ follows from Theorem \ref{r=1 case}. Now take any $n\ge 3$. Since $G^c$ is a chordal graph, without loss of generality we can assume that $x_1$ is a simplicial vertex of $G^c$. Then we can write 
    \[
    I_r(G)=x_1J_1+J_2,
    \]
    where 
    \[J_1=\left\langle \prod_{s=1}^r x_{i_s}\mid G[\{x_1,x_{i_1},\ldots,x_{i_r}\}] \text{ is a connected subgraph of } G \right\rangle,\] 
    and the ideal 
    \[J_2=\left\langle \prod_{s=1}^{r+1} x_{j_s}\mid x_1\notin \{x_{j_1},\ldots,x_{j_{r+1}}\}\text{ and }G[\{x_{j_1},\ldots,x_{j_{r+1}}\}] \text{ is connected }  \right\rangle.\] 
    Now construct the graph $\widetilde G$ from $G$ as described above, i.e., 
\begin{align*}
    V(\widetilde G)&=V(G)\setminus\{x_1\};\\
    E(\widetilde G)&=E(G\setminus x_1)\cup \{\{x_i,x_j\}\mid \{x_i,x_j\}\in E(G^c\setminus x_1)\text{ and } N_{G^c}(x_1)\cap\{x_i,x_j\}=\emptyset\}.
\end{align*}
Rename the vertices of $G$ as before:  $V(G)=V(G^c)=\{x_1\}\cup Y\cup W$, where $Y=N_{G^c}(x_1)$ and $W=N_G(x_1)$. Then $\{x_1\}\cup Y$ is an independent set in $G$. Moreover, $x_1$ is connected by an edge in $G$ to all the vertices in $W$. In $\widetilde G$ the vertices in $Y$ forms an independent set and the the induced subgraph on $W$ is a clique. Thus in $\widetilde G^c$ the vertices in $Y$ forms a clique, whereas the vertices in $W$ forms an independent set. By Lemma \ref{aux lemma 1}, $\widetilde G$ is a co-chordal graph.
Now our aim is to prove the following.

\begin{enumerate}[$(i)$]
    \item $J_2=I_r(G\setminus x_1)$,
    \item $J_2\subseteq J_1$, and
    \item $J_1=I_{r-1}(\widetilde G)$.
\end{enumerate}

\begin{proof}[Proof of (i)]
 If $\mathbf x_S\in J_2$ is a minimal
 generator, then $x_1\notin S$ and $G[S]$ is connected, where $|S|=r+1$. Therefore, $\mathbf x_S\in I_r(G\setminus x_1)$. Similarly, if $\mathbf x_S\in I_r(G\setminus x_1)$ is a minimal generator, then $x_1\notin S$ and hence $\mathbf x_{S}\in J_2$. Thus $J_2=I_r(G\setminus x_1)$.    
\end{proof}

\begin{proof}[Proof of (ii)]
 Let $\mathbf x_S\in J_2$. Then $x_1\notin S$ and $G[S]$ is connected, where $|S|=r+1$. If $Y\cap S=\emptyset$, then for any $w\in W\cap S$, the induced subgraph on $\{x_1\}\cup(S\setminus\{w\} )$ is a connected subgraph of $G$ since $N_G(x_1)=W$. Note here that $|S\setminus\{w\}|=r$ and $\mathbf x_{S\setminus\{w\}}\in J_1$. Since $\mathbf x_{S\setminus\{w\}}\vert \mathbf x_{S}$, we have $\mathbf x_S\in J_1$. If $Y\cap S\neq \emptyset$, then for any $y\in Y$, the induced subgraph on $\{x_1\}\cup(S\setminus \{y\})$ is a connected subgraph of $G$. Indeed, if $z \in (S \setminus 
 \{y\}) \cap W$, then $z$ is adjacent to $x_1$.
 If $y' \in (S \setminus \{y\}) \cap Y$, then
 there is $w \in W$ such that $\{y',w\} \in E(G)$
 since $y'$ is connected to some element of $S$,
 but it cannot be $y$ since the set $Y$ is an
 independent set.  But because $\{w,x_1\} \in E(G)$, there is a path for all vertices
 $(S \setminus \{y\})$ to $x_1$.
 Consequenty, $\mathbf x_{S\setminus \{y\}}\in J_1$, where $|S\setminus \{y\}|=r$. Hence $\mathbf x_S\in J_1$. Thus we have $J_2\subseteq J_1$. 
\end{proof}

\begin{proof}[Proof of (iii)]
 Let $\mathbf x_S\in J_1$. Then $|S|=r$ and $G[S\cup \{x_1\}]$ is a connected subgraph of $G$. Since $\{x_1\}\cup Y$ forms an independent set in $G$, we have $S\cap W\neq \emptyset$. Moreover, for each $y\in Y\cap S$, there exists some $w\in S\cap W$ such that $\{y,w\}\in E(G)$. Therefore, $\widetilde G[S]$ is connected as $W$ forms a clique in $\widetilde G$. Thus $\mathbf x_S\in I_{r-1}(\widetilde G)$ and hence $J_1\subseteq I_{r-1}(\widetilde G)$. Now let $\mathbf x_S\in I_{r-1}(\widetilde G)$. Then $|S|=r$ and $\widetilde G[S]$ is connected. Since $Y$ forms an independent set in $\widetilde G$, we see that for each $y\in Y\cap S$, there exists some $w\in S\cap W$ such that $\{y,w\}\in E(\widetilde G)$. Now $x_1$ is connected by an edge in $G$ to all the vertices in $W$. Hence $G[S\cup\{x_1\}]$ is connected and consequently, $I_{r-1}(\widetilde G)\subseteq J_1$. This completes the proof of (iii).    
\end{proof}

Now by induction on $n$, $J_2$ is a vertex splittable ideal since $G\setminus x_1$ is
co-chordal on fewer vertices. Also, since $J_1=I_{r-1}(\widetilde G)$ and $\widetilde G$ is a co-chordal graph (by Lemma \ref{aux lemma 1}), we see that $J_1$ is a vertex splittable ideal by induction on $r$. If $J_2=\langle 0 \rangle$, then by Lemma \ref{extra variable product}, $I_r(G)=x_1J_1$ is a vertex splittable ideal. Otherwise, by Definition \ref{vertex splittable ideal}, we see that $I_r(G)=x_1J_1+J_2$ is a vertex splittable ideal. 
\end{proof}

A {\it claw} is a graph on four vertices $\{x,y_1,y_2,y_3\}$ with edge set $\{\{x,y_1\},\{x,y_2\},\{x,y_3\}\}$. 
We say that a graph is {\it claw free} if it does not contain any claw as an induced subgraph. 
A graph $G$ is said to be {\it gap free} if $G^c$ does not contain $C_4$ as an induced subgraph.

\begin{corollary}\label{gap free corollary}
    Let $G$ be a gap free and claw free graph such that it contains a leaf. Then both $I_r(G)$ and $I_r(G^c)$ are vertex splittable ideals for any integer $r\ge 1$.
\end{corollary}

\begin{proof}
    Let $x\in V(G)$ be a leaf of $G$ and let $N_G(x)=\{y\}$. If $N_G(y)=\{x\}$, then for any $u,v\in V(G)\setminus\{y,x\}$, $\{u,v\}\notin E(G)$. Because if $\{u,v\}\in E(G)$, then $\{y,x,u,v\}$ forms a cycle of length four in $G^c$, a contradiction. Thus $V(G)\setminus\{x,y\}$ forms an independent set in $G$. Therefore, by Proposition \ref{proposition 1}, both $G$ and $G^c$ are chordal. Hence by Theorem \ref{general r case}, $I_r(G)$ and $I_r(G^c)$ are vertex splittable ideals for $r\ge 1$. 
    
    Now let $N_G(y)=\{x,y_1,\ldots,y_t\}$ for some $t\ge 1$. For $i\neq j$, if $\{y_i,y_j\}\notin E(G)$ then $\{x,y,y_i,y_j\}$ forms a claw in $G$, which is not possible. Hence $\{y,y_1,\ldots,y_t\}$ forms a clique in $G$. Now let $a,b\in V(G)\setminus N_G[y]$. If $\{a,b\}\in E(G)$, then $\{a,b,x,y\}$ forms a cycle of length $4$ in $G^c$, a contradiction. Hence $V(G)\setminus N_G[y]$ forms an independent set in $G$. Consequently, $\{x\}\cup(V(G)\setminus N_G[y])$ forms an independent set in $G$. Let $Y=N_G[y]\setminus \{x\}$ and $W=\{x\}\cup (V(G)\setminus N_G[y])$. Then $V(G)=Y\sqcup W$ such that $G[Y]$ is a clique and $W$ forms an independent set in $G$. Therefore, by Proposition \ref{proposition 1}, both $I_r(G)$ and $I_r(G^c)$ are vertex splittable ideals for $r\ge 1$.
\end{proof}

A vertex splittable ideal has a linear resolution (by \cite[Theorem 2.4]{SF1}). Thus, as an application of Theorem \ref{general r case}, we see that one direction of Fr\"{o}berg's theorem is true in the context of Stanley-Reisner ideals of $r$-independence complexes of graphs:

\begin{corollary}\label{cor.froberg1}
    If $G$ is a co-chordal graph, then $I_r(G)$ has a $(r+1)$-linear resolution.
\end{corollary}

\begin{remark}
The converse of \Cref{cor.froberg1} is not true in general. For example, if $r\ge |V(G)|$ and $G$ is any graph, then the ideal $I_{\mathrm{Ind}_r(G)}$ is the zero ideal and hence has a linear resolution. Moreover, if $r=|V(G)|-1$ then $I_{\mathrm{Ind}_r(G)}$ is the zero ideal if $G$ is not connected and is generated by the monomial $\prod_{x_i\in V(G)}x_i$ if $G$ is connected. Thus in both cases $I_{\mathrm{Ind}_r(G)}$ has a linear resolution. 

For another example, let $C_n=x_1x_2\cdots x_n$ be a cycle of length $n$. Consider the graph $H=C_n*\{x_{n+1}\}$, the graph formed by adding a new vertex $x_{n+1}$ and joining this vertex by an edge to all vertices in the cycle. Then $H$ is not chordal but for $n\ge 5$, $I_{\mathrm{Ind}_{n-1}(G)}=\langle x_1x_2\cdots x_n\rangle$ has a linear resolution, where $G=H^c$. In general, if $G$ is any graph such that its connected components have cardinality at most $s$, then for each $r\ge s$, $I_{\mathrm{Ind}_r(G)}$ is the zero ideal and hence has a linear resolution. Moreover, if $\mathcal C_1,\ldots,\mathcal C_k$ are connected components of $G$ such that $|V(\mathcal C_i)|=s$ for some $i$ and $|V(\mathcal C_j)|<s$ for each $j\neq i$, then $I_{\Indr{s-1}(G)}=\langle \prod_{x_i\in V(\mathcal C_i)}x_i \rangle$ and hence has a linear resolution.
\end{remark}

\section{The graded Betti numbers}\label{Betti numbers}

The vertex splittability property of a square-free monomial ideal gives important information about the graded Betti numbers. Namely, the Betti numbers of the ideal can be expressed as a sum of the Betti numbers of some `smaller' ideals. More precisely, we have the following theorem by Moradi and Khosh-Ahang. 

\begin{theorem}\textup{\cite[Theorem 2.8]{SF1}}\label{Betti splitting}
    Let $I = xJ_1+J_2\subseteq R$ be a vertex splitting for the monomial ideal $I$. Then the graded Betti numbers
of $R/I$ can be computed by the following recursive formula:
\[
    \beta_{i,j} (R/I ) = \beta_{i,j-1}(R/J_1) + \beta_{i,j} (R/J_2) + \beta_{i-1,j-1}(R/J_2)
~~\mbox{for all $i,j \geq 0$.}\]
\end{theorem}

Using Theorem \ref{Betti splitting}, we can compute all the graded Betti numbers of Stanley-Reisner ideals of higher independence complexes of some well-known families of graphs, namely, the complete graphs, star graphs and the complement of path graphs. But before proceeding to derive the formulas we first make the following remark.

\begin{remark}
    In the calculations below, we make extensive use of Pascal's identity:
    \[
    \binom{n}{r} + \binom{n}{r-1} = \binom{n+1}{r}.
    \]
\end{remark}

The following theorem is well-known using the properties of complete intersections. 
\begin{theorem}\label{r=0 case}
    Let $I = \langle x_{i_1},\ldots,
    x_{i_k} \rangle \subset R$ be a monomial ideal generated by $k$ variables. Then the $\mathbb N$-graded Betti numbers of $R/I$ are given by the following formula: $\beta_{i,j}(R/I) =0$
    if $i \neq j$, and 
    \[
    \beta_{i,i}(R/I)=\begin{cases}
        \binom{k}{i}&\text{ for }0\le i\le k\\
        0 &\text{ otherwise.}
    \end{cases}
    \]
\end{theorem}

 We now give explicit formulas for all the Betti numbers of the $r$-independence complexes of complete graphs. Note that the formula is well-known in the case of $r=1$ and it was first derived in the thesis of Jacques \cite[Theorem 5.1.1]{SJ}. 
 We provide two different proofs; one using the formula in \Cref{Betti splitting} and another using the Hilbert series of the ideal.

\begin{theorem}\label{complete graph case}
    Let $I_r(K_n)$ denote the Stanley-Reisner ideal of the $r$-independence complex of $K_n$ for $r\ge 1$. Then the $\mathbb N$-graded Betti numbers of $R/I_r(K_n)$ can be expressed as follows:  $\beta_{i,j}(R/I_r(K_n)) =0$ if
    $j \neq i+r$, and 
\[
\beta_{i,i+r}(R/I_r(K_n))=\begin{cases}
    \binom{i+r-1}{r}\binom{n}{i+r} &\text{ for }1\le i\le n-r\\
    0 &\text{ otherwise.}
\end{cases}
\]  
\end{theorem}
\begin{proof}
    We prove this by induction on $n$ and $r$. For a fixed $r$, if $n\in \{1,2,\ldots,r\}$ then $I_r(K_n)=\langle 0\rangle$. Hence we have the above formula. Now fix an $n$. We first prove the above formula for $r=1$. By Theorem \ref{r=1 case} we have 
    \[
    I_1(K_n)=x_1J_1+J_2,
    \]
where $J_1=\langle x_2,\ldots,x_n \rangle$ and $J_2=I_1(K_{n}\setminus\{x_1\})$. Note that $K_{n}\setminus\{x_1\}=K_{n-1}$. Therefore, using Theorem \ref{Betti splitting} and by the induction on $n$, we have, for $1\le i\le n-1$,
\begin{align*}
\beta_{i,i+1}(R/I_1(K_n))=& \beta_{i,i}(R/J_1)+\beta_{i,i+1}(R/I_1(K_{n-1}))+\beta_{i-1,i}(R/I_1(K_{n-1}))\\
&=\binom{n-1}{i}+i\binom{n-1}{i+1}+(i-1)\binom{n-1}{i}\hspace{1cm}(\text{by Theorem \ref{r=0 case}})\\
&=\binom{i}{1}\binom{n}{i+1}.
\end{align*}

\noindent
Now we take any $r\ge 2$. Then by the construction in Theorem \ref{general r case},
\begin{align*}
    I_r(K_n)=x_1J_1+J_2,
\end{align*}
where $J_1=I_{r-1}(\widetilde{K_{n}})$ and $J_2=I_{r}(K_{n-1})$. Observe that $\widetilde{K_{n}}=K_{n-1}$. Hence using Theorem \ref{Betti splitting} and by the induction on $n$ and $r$, we have, for $1\le i\le n-r$, 
\begin{align*}
    &\beta_{i,i+r}(R/I_r(K_n))\\
    &=\beta_{i,i+r-1}(R/I_{r-1}(K_{n-1}))+\beta_{i,i+r}(R/I_{r}(K_{n-1}))+\beta_{i-1,i+r-1}(R/I_{r}(K_{n-1}))\\
    &=\binom{i+r-2}{r-1}\binom{n-1}{i+r-1}+\binom{i+r-1}{r}\binom{n-1}{i+r}+\binom{i+r-2}{r}\binom{n-1}{i+r-1}\\
    &=\binom{i+r-1}{r}\binom{n}{i+r}
\end{align*}
\end{proof}

\begin{remark}
    The ideal $I_r(K_n)$ is same as the edge ideal of the $(r+1)$-complete hypergraph $K_n^{r+1}$ defined by Emtander \cite{Emtander}. Thus Theorem \ref{complete graph case} shows that the edge ideal of the complete hypergraph $K_n^{r+1}$ is vertex splittable.
\end{remark}

 \begin{example}
   Let $G=K_7$ and $r=3$. Then the Betti table of $R/I_3(K_7)$ computed using Theorem \ref{complete graph case} as follows.
    \begin{figure}[!ht]\
\centering
\begin{tikzpicture}
[scale=.45]
\draw [fill] (4,0) circle [radius=0.1];
\draw [fill] (7,0) circle [radius=0.1];
\draw [fill] (9,2) circle [radius=0.1];
\draw [fill] (9,4) circle [radius=0.1];
\draw [fill] (2,2) circle [radius=0.1];
\draw [fill] (2,4) circle [radius=0.1];
\draw [fill] (5.5,7) circle [radius=0.1];
\node at (4,-0.8){$x_1$};
\node at (7,-0.8){$x_2$};
\node at (10,2){$x_3$};
\node at (10,4){$x_4$};
\node at (5.5,8){$x_5$};
\node at (1,4){$x_6$};
\node at (1,2){$x_7$};
\draw (4,0)--(7,0)--(9,2)--(2,2)--(2,4)--(5.5,7)--(4,0)--(9,2)--(2,2)--(5.5,7)--(7,0)--(9,4)--(2,4)--(4,0)--(9,4)--(2,2)--(4,0)--(2,2)--(7,0)--(2,4)--(9,2)--(5.5,7)--(9,4)--(9,2);

\draw [fill] (16,7) circle [radius=0.01];
\draw [fill] (16,8.5) circle [radius=0.01];
\draw (16,7)--(27,7);
\draw (16,8.5)--(27,8.5);
\node at (18,7.8){$0$};
\node at (19.4,7.8){$1$};
\node at (21.2,7.8){$2$};
\node at (23.3,7.8){$3$};
\node at (25.5,7.8){$4$};
\node at (16.5,6.1){$0:$};
\node at (16.5,4.4){$1:$};
\node at (16.5,2.6){$2:$};
\node at (18,6){$1$};
\node at (19.4,6){$\cdot$};
\node at (21.2,6){$\cdot$};
\node at (23.3,6){$\cdot$};
\node at (25.5,6){$\cdot$};
\node at (18,4.3){$\cdot$};
\node at (19.4,4.3){$\cdot$};
\node at (21.2,4.3){$\cdot$};
\node at (23.3,4.3){$\cdot$};
\node at (25.5,4.3){$\cdot$};
\node at (18,2.5){$\cdot$};
\node at (19.4,2.5){$\cdot$};
\node at (21.2,2.5){$\cdot$};
\node at (23.3,2.5){$\cdot$};
\node at (25.5,2.5){$\cdot$};
\node at (16.5,0.8){$3:$};
\node at (18,0.7){$\cdot$};
\node at (19.4,0.7){$35$};
\node at (21.2,0.7){$84$};
\node at (23.3,0.7){$70$};
\node at (25.5,0.7){$20$};
\end{tikzpicture}
\caption{ The graph $K_7$ and the Betti table of $R/I_3(K_7)$}
\end{figure}
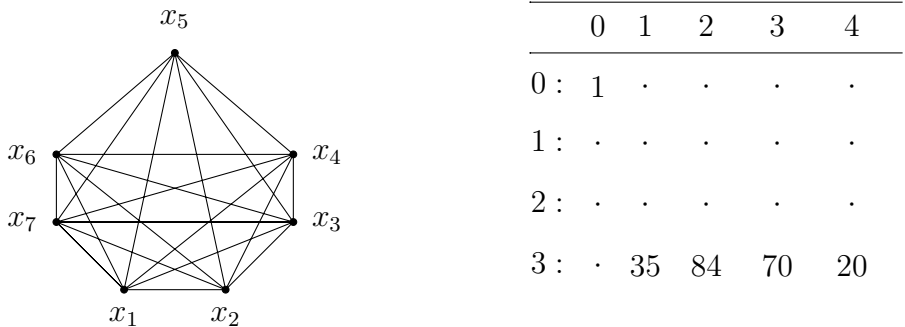
 \end{example}

Recall that, if $M$ is an $\mathbb N$-graded $R$ module, then the Hilbert series $H_M(t)$ measures the $\mathbb K$-vector space dimensions of the graded pieces $M_i$ of $M$. More specifically, if each graded piece $M_i$ has finite $\mathbb K$-vector space dimension, then $H_M(t)$ is the formal power series 
 \[
 H_M(t)=\sum_{i\in\mathbb N}\dim_{\mathbb K}(M_i)t^i.
 \]

The following result is well-known. See, for example, \cite[Section 6.1.3]{HH}.

\begin{lemma}
    Let $R$ be the polynomial ring $\mathbb K[x_1,\ldots,x_n]$, and consider a finitely generated $\mathbb N$-graded $R$-module $M$. Then
    \[
    H_M(t)=\frac{R_M(t)}{(1-t)^n},
    \]
    where $R_M(t)=\sum_{i}(-1)^i\sum_j\beta_{i,j}(M)t^j$.
\end{lemma}
 
Note that, if $M$ is the Stanley–Reisner ring of a simplicial complex $\Delta$, then the Hilbert series of $M$ can be calculated from combinatorial data of $\Delta$.
In particular, we have the following formula (see \cite[Chapter 6]{Vill}):
\[
H_{R/I_{\Delta}}=\frac{1}{(1-t)^n}\sum_{s=0}^l f_{s-1} t^s (1-t)^{n-s},
\]
where $f_{s}$ equals the number of $s$-dimensional faces of $\Delta$ and $l=\dim\Delta+1$.

If $R/I_{\Delta}$ has a linear minimal free resolution, then the Betti numbers of $R/I_{\Delta}$ can be easily deduced from the Hilbert series. 
Let $G$ be a co-chordal graph, then $R/I_r(G)$ has a linear minimal free resolution (by \Cref{cor.froberg1}). 
Thus by \cite[Lemma 3.11]{Emtander} we get the following formula.

 \begin{lemma}\label{Hilbert series}
      If $G^c$ is a chordal graph, then the Betti numbers $\beta_{i,i+r}(R/I_r(G))$ are given by the formula
     \begin{align}\label{alternate formula}
     \beta_{i,i+r}(R/I_r(G))=\sum_{s=0}^l(-1)^{r-s}f_{s-1}\binom{n-s}{i+r-s},
     \end{align}
     where $f_{s}$ equals the number of $s$-dimensional faces of $\Indr{r}(G)$ and $l=\dim\Indr{r}(G)+1$. 
 \end{lemma}

{Lemma \ref{Hilbert series},
like Theorem \ref{Betti splitting}, provides us with a tool to compute the graded Betti
numbers of $R/I_r(G)$ when $G$ is co-chordal.  By playing the two results off each other, we may be able to derive some combinatorial identities.  
}

{As a specific example, if 
we take $G=K_{n}$, then $f_s(\Indr{r}(K_n))=\binom{n}{s+1}$ and $l=r$.  Thus
\[\beta_{i,i+r}(R/I_r(G))=
\sum_{s=0}^{r}(-1)^{r-s}\binom{n}{s}\binom{n-s}{i+r-s}.
 \]
 We also computed these graded
Betti numbers in \Cref{complete graph case}
by using Theorem \ref{Betti splitting}.
Thus, since these two expressions are equal,
we can derive the combinatorial identity:
 \[
 \binom{i+r-1}{r}\binom{n}{i+r}=\sum_{s=0}^{r}(-1)^{r-s}\binom{n}{s}\binom{n-s}{i+r-s}.
 \]
See also \cite[Section 3.3]{Emtander} for a combinatorial proof.}

If $G=K_{n_1,\ldots,n_s}$, the complete multipartite graph on $N_s:=\sum_{t=1}^sn_t$ number of vertices, then the ideal $I_r(G)$ is same as the edge ideal of 
the $(r+1)$-complete multipartite hypergraph $K_{n_1,\ldots,n_s}^{r+1}$ as defined by Emtander \cite[Section 3.1]{Emtander}. 
Hence, by Corollary \ref{cor.froberg1} we see that the edge ideal of the $(r+1)$-complete multipartite hypergraph $K_{n_1,\ldots,n_s}^{r+1}$ is vertex splittable. 
Also, we have the formula for the Betti numbers of $I_r(K_{n_1,\ldots,n_s})$.

 \begin{theorem}\textup{\cite[cf. Theorem 3.5]{Emtander}}\label{bipartite}
Let $G = K_{n_1,n_2,\ldots,n_s}$ denote the complete multipartite graph on $N_s$ number of vertices. 
Then for all positive integers $r$, the $\mathbb N$-graded Betti numbers of the ideal $I_r(G)$ can be expressed as follows:
    {\footnotesize
    \[
\beta_{i,i+d}(R_{G}/I_r(G))=\begin{cases}
    \sum_{t=1}^s\binom{n_t}{i+r}\binom{i+r-1}{r}-\sum_{J(r,s)} \left[\prod_{t=1}^s\binom{n_t}{j_t}\right]\sum_{t=1}^s\binom{j_t-1}{r}&\text{ for }d=r,\\
    0 &\text{ otherwise.}
\end{cases}
\]  }
For brevity we let $J(r,s)$ denote all those tuples $(j_1,j_2,\ldots,j_s)\in\mathbb N^s$ such that $\sum_{t=1}^sj_t=i+r$.
\end{theorem}
 Let $P_n^c$ denote the complement of path graph on $n$ vertices. We now proceed to give formulas for the Betti numbers of $R/I_r(P_n^c)$ using the vertex splittability of the ideal $I_r(P_n^c)$. Note that if $r=1$ then $I_1(P_n^c)$ is the edge ideal of $P_n^c$. To the best of our knowledge, the formula for $\beta_{i,j}(R/I_1(P_n^c))$ in \Cref{path complement formula} is new.
 
 In order to find the Betti numbers of $R/I_r(P_n^c)$, we first analyze the graph $K_n^x$ defined as follows:
\begin{align*}
    V(K_n^x)&=\{x,x_1,x_2,\ldots,x_n\}\\
    E(K_n^x)&=\{\{\{x_i,x_j\}\mid 1\le i<j\le n\}\cup\{\{x,x_i\}\mid i\in [n-1]\}.
\end{align*}
Note that the graph $K_n^x$ is the 
graph $K_{n+1}$ on the vertex set $\{x,x_1,\ldots,x_n\}$ with the edge $\{x,x_n\}$ removed.
For $r\ge 1$, consider the ideal $I_r(K_n^x)$ in the polynomial ring $R_x=\mathbb K[x,x_1,x_2,\ldots,x_n]$.

\begin{lemma}\label{lemma for path complement}
    The $\mathbb N$-graded Betti numbers of $R_x/I_r(K_n^x)$ can be expressed as follows:
\[
\beta_{i,i+r}(R_x/I_r(K_n^x))=\begin{cases}
    i\binom{n+1}{i+1}-\binom{n-1}{i-1}, &\text{ for }r=1\text{ and }1\le i\le n,\\
    \binom{i+r-1}{r}\binom{n+1}{i+r}, &\text{ for }r\ge 2\text{ and }1\le i\le n-r+1,\\
    0, &\text{otherwise}.
\end{cases}
\]
\end{lemma}
\begin{proof}
 Note that $K_n^x$ is a co-chordal graph with $x$ as a simplicial vertex of the complement of $K_n^x$. Thus by 
Theorem \ref{general r case}, $I_r(K_n^x)$ is a vertex splittable ideal. 
Let $r=1$. Then
\[
I_1(K_n^x)=xJ_1+J_2,
\]
where $J_1=\langle x_1,\ldots,x_{n-1}\rangle$ and $J_2=I_1(K_n^x\setminus\{x\})$. Note that $K_n^x\setminus\{x\}=K_n$. Therefore, by Theorem \ref{Betti splitting},
\begin{align*}
    \beta_{i,i+1}(R_x/I_1(K_n^x))&= \beta_{i,i}(R_x/J_1)+\beta_{i,i+1}(R_x/I_1(K_{n}))+\beta_{i-1,i}(R_x/I_1(K_{n}))\\
    &=\binom{n-1}{i}+i\binom{n}{i+1}+(i-1)\binom{n}{i}\\
    &=i\binom{n+1}{i+1}-\binom{n-1}{i-1}.
\end{align*}
Now let $r\ge 2$. Then by the construction in Theorem \ref{general r case},
 \[
 I_r(K_n^x)=xJ_1+J_2,
 \]
where $J_1=I_{r-1}(\widetilde{K_n^x})$ and $J_2=I_r(K_n^x\setminus\{x\})$. Note that $\widetilde{K_n^x}=K_n^x\setminus\{x\}=K_n$. Therefore, by Theorem \ref{Betti splitting},
 \begin{align*}
     &\beta_{i,i+r}(R_x/I_r(K_n^x))\\
     &= \beta_{i,i+r-1}(R_x/I_{r-1}(K_{n}))+\beta_{i,i+r}(R_x/I_{r}(K_{n}))+\beta_{i-1,i+r-1}(R_x/I_r(K_{n}))\\
     &=\binom{i+r-2}{r-1}\binom{n}{i+r-1}+\binom{i+r-1}{r}\binom{n}{i+r}+\binom{i+r-2}{r}\binom{n}{i+r-1}\\
     &=\binom{i+r-1}{r}\binom{n+1}{i+r}.
 \end{align*}
 
 \end{proof}

 \begin{remark}
     Note that $\beta_{i,i+r}(R_x/I_r(K_n^x))=\beta_{i,i+r}(R/I_r(K_{n+1}))$ for $r\ge 2$.
 \end{remark}

 Now we are ready to give explicit formulas for the Betti numbers $\beta_{i,i+r}(R/I_r(P_n^c))$.

 \begin{theorem}\label{path complement formula}
     Let $I_r(P_n^c)$ denote the Stanley-Reisner ideal of the $r$-independence complex of the complement of the path graph $P_n$ on $n$ vertices $\{x_1,\ldots,x_n\}$. Then the $\mathbb N$-graded Betti numbers of $R/I_r(P_n^c)$ can be expressed as follows:  $\beta_{i,j}(R/I_r(P_n^c))=0$
     if $j \neq i+r$, and 
     \begin{align*}
         \beta_{i,i+r}(R/I_r(P_n^c))=\begin{cases}
             i\binom{n-1}{i+1},&\text{ for }r=1\text{ and }1\le i\le n-2,\\
             \binom{i+1}{2}\binom{n}{i+2}-i\binom{n-2}{i},&\text{ for }r=2\text{ and }1\le i\le n-2,\\
             \binom{i+r-1}{r}\binom{n}{i+r},&\text{ for }r\ge 3\text{ and }1\le i\le n-r,\\
             0,&\text{ otherwise.}
         \end{cases}
     \end{align*}
 \end{theorem}
 \begin{proof}
     The graph $P_n^c$ is a co-chordal graph with $x_1$ as a simplicial vertex of $P_n$. We first consider the case $r=1$. By Theorem \ref{r=1 case},
     \[
     I_1(P_n^c)=x_1J_1+J_2,
     \]
 where $J_1=\langle x_3,x_4,\ldots,x_n \rangle$ and $J_2=I_1(P_{n}^c\setminus\{x_1\})$. Thus $J_2=I_1(P_{n-1}^c)$ since $P_{n}^c\setminus\{x_1\}=P_{n-1}^c$. We now proceed to prove the formula by induction on $n$. For $n=1,2$, $I_1(P_n^c)=\langle 0\rangle$ and thus we have the above formula. For $n=3$, $I_1(P_n^c)=\langle x_1x_3 \rangle=I_1(K_2)$ and hence using Theorem \ref{complete graph case} we have our desired formula. Now let $n\ge 4$. Then by Theorem \ref{Betti splitting} and using the induction hypothesis we have,
\begin{align*}
    \beta_{i,i+1}(R/I_r(P_n^c))&=\beta_{i,i}(R/J_1)+\beta_{i,i+1}(R/I_1(P_{n-1}^c))+\beta_{i-1,i}(R/I_1(P_{n-1}^c))\\
    &=\binom{n-2}{i}+i\binom{n-2}{i+1}+(i-1)\binom{n-2}{i}\\
    &=i\binom{n-1}{i+1}.
\end{align*}
Now we consider the case $r=2$. By Theorem \ref{general r case} we have,
\[
I_2(P_n^c)=x_1J_1+J_2,
\]
 where $J_1=I_1(\widetilde{P_n^c})$ and $J_2=I_2(P_{n-1}^c)$. Note that $\widetilde{P_n^c}\cong K_{n-2}^x$. Now, for $n=1,2,3$, we have $I_1(\widetilde{P_n^c})=I_2(P_{n-1}^c)=\langle 0\rangle$. For $n=4$, we have $J_1=I_1(K_2^x)$ and $J_2=I_2(P_3^c)=\langle 0\rangle$. Therefore, by Theorem \ref{Betti splitting} and Lemma \ref{lemma for path complement},
\begin{align*}
\beta_{i.i+2}(R/I_2(P_4^c))&=\beta_{i,i+1}(R/I_1(K_2^x))+\beta_{i,i+1}(R/I_2(P_{3}^c))+\beta_{i-1,i}(R/I_2(P_{3}^c))\\
&=i\binom{3}{i+1}-\binom{1}{i-1}.
\end{align*}

\noindent
It is not difficult to see that for all non-negative integers $i$, $i\binom{3}{i+1}-\binom{1}{i-1}=\binom{i+1}{2}\binom{4}{i+2}-i\binom{2}{i}$. Thus we have $\beta_{i,i+2}(R/I_2(P_4^c))=\binom{i+1}{2}\binom{4}{i+2}-i\binom{2}{i}$. Now let $n\ge 5$. Then by Theorem \ref{Betti splitting} and by the induction hypothesis,
{\small
\begin{align*}
    &\beta_{i,i+2}(R/I_2(P_n^c))\\
    &=\beta_{i,i+1}(R/I_1(K_{n-2}^x))+\beta_{i,i+1}(R/I_2(P_{n-1}^c))+\beta_{i-1,i}(R/I_2(P_{n-1}^c))\\
    &=i\binom{n-1}{i+1}-\binom{n-3}{i-1}+\binom{i+1}{2}\binom{n-1}{i+2}-i\binom{n-3}{i}+\binom{i}{2}\binom{n-1}{i+1}-(i-1)\binom{n-3}{i-1}\\
    &=\binom{i+1}{2}\binom{n}{i+2}-i\binom{n-2}{i}.
\end{align*}
}

We now take $r\ge 3$ and prove the formula by induction on $n$ and $r$. By Theorem \ref{general r case},
\[
I_r(P_n^c)=x_1J_1+J_2,
\]
where $J_1=I_{r-1}(K_{n-2}^x)$ and $J_2=I_r(P_{n-1}^c)$. Note that for a fixed $r$, $I_r(P_n^c)=\langle 0\rangle$ for $n=1,2,\ldots,r$. If $n=r+1$, then $I_r(P_n^c)= \langle \prod_{i=1}^{r+1}x_i \rangle$. Hence $\beta_{i,i+r}(R/I_r(P_{n}^c))=\binom{i+r-1}{r}\binom{r+1}{i+r}$ since $$\binom{i+r-1}{r}\binom{r+1}{i+r}=\begin{cases}
    1&\text{ if }i=1\\
    0& \text{ otherwise.}
\end{cases}$$ Now let $n\ge r+2$. Then by Theorem \ref{Betti splitting}, Lemma \ref{lemma for path complement} and by the induction on $n$ and $r$,
\begin{align*}
     &\beta_{i,i+r}(R/I_r(P_n^c))\\
      &=\beta_{i,i+r-1}(R/I_{r-1}(K_{n-2}^x))+\beta_{i,i+r}(R/I_r(P_{n-1}^c))+\beta_{i-1,i+r-1}(R/I_r(P_{n-1}^c))\\
      &=\binom{i+r-2}{r-1}\binom{n-1}{i+r-1}+\binom{i+r-1}{r}\binom{n-1}{i+r}+\binom{i+r-2}{r}\binom{n-1}{i+r-1}\\
      &=\binom{i+r-1}{r}\binom{n}{i+r}.
\end{align*}
 \end{proof}

  \begin{example}
   Let $G=P_7^c$. Then the Betti tables of $R/I_2(P_7^c)$ and 
$R/I_3(P_7^c)$ computed using Theorem \ref{complete graph case} are given as 
in Figure \ref{figurebettitable}.
    \begin{figure}[!ht]\
\centering
\begin{tikzpicture}
[scale=.45]
\draw [fill] (1,7) circle [radius=0.01];
\draw [fill] (1,8.5) circle [radius=0.01];
\draw (1,7)--(13,7);
\draw (1,8.5)--(13,8.5);
\node at (3,7.8){$0$};
\node at (4.4,7.8){$1$};
\node at (6.2,7.8){$2$};
\node at (8.3,7.8){$3$};
\node at (10.5,7.8){$4$};
\node at (12.7,7.8){$5$};
\node at (1.5,6.1){$0:$};
\node at (1.5,4.4){$1:$};
\node at (1.5,2.6){$2:$};
\node at (3,6){$1$};
\node at (4.4,6){$\cdot$};
\node at (6.2,6){$\cdot$};
\node at (8.3,6){$\cdot$};
\node at (10.5,6){$\cdot$};
\node at (12.7,6){$\cdot$};
\node at (3,4.3){$\cdot$};
\node at (4.4,4.3){$\cdot$};
\node at (6.2,4.3){$\cdot$};
\node at (8.3,4.3){$\cdot$};
\node at (10.5,4.3){$\cdot$};
\node at (12.7,4.3){$\cdot$};
\node at (3,2.5){$\cdot$};
\node at (4.4,2.5){$30$};
\node at (6.2,2.5){$85$};
\node at (8.3,2.5){$96$};
\node at (10.5,2.5){$50$};
\node at (12.7,2.5){$10$};
\draw [fill] (16,7) circle [radius=0.01];
\draw [fill] (16,8.5) circle [radius=0.01];
\draw (16,7)--(27,7);
\draw (16,8.5)--(27,8.5);
\node at (18,7.8){$0$};
\node at (19.4,7.8){$1$};
\node at (21.2,7.8){$2$};
\node at (23.3,7.8){$3$};
\node at (25.5,7.8){$4$};
\node at (16.5,6.1){$0:$};
\node at (16.5,4.4){$1:$};
\node at (16.5,2.6){$2:$};
\node at (18,6){$1$};
\node at (19.4,6){$\cdot$};
\node at (21.2,6){$\cdot$};
\node at (23.3,6){$\cdot$};
\node at (25.5,6){$\cdot$};
\node at (18,4.3){$\cdot$};
\node at (19.4,4.3){$\cdot$};
\node at (21.2,4.3){$\cdot$};
\node at (23.3,4.3){$\cdot$};
\node at (25.5,4.3){$\cdot$};
\node at (18,2.5){$\cdot$};
\node at (19.4,2.5){$\cdot$};
\node at (21.2,2.5){$\cdot$};
\node at (23.3,2.5){$\cdot$};
\node at (25.5,2.5){$\cdot$};
\node at (16.5,0.8){$3:$};
\node at (18,0.7){$\cdot$};
\node at (19.4,0.7){$35$};
\node at (21.2,0.7){$84$};
\node at (23.3,0.7){$70$};
\node at (25.5,0.7){$20$};
\end{tikzpicture}
\caption{Betti table of $R/I_2(P_7^c)$ and $R/I_3(P_7^c)$, respectively}\label{figurebettitable}
\end{figure}
 \end{example}

 \begin{remark}
     Note that if we compare the formulas for the graphs $P_n^c$, $K_n$ and $K_{1,n}$, for $r\ge 3$, we have $\beta_{i,i+r}(R/I_r(P_n^c))=\beta_{i,i+r}(R/I_r(K_n))= \beta_{i,i+r+1}(R/I_r(K_{1,n}))$ for all $i \geq 0$.
 \end{remark}

 \begin{remark}
     Note that for $r\ge 3$, $I_r(P_n^c)=I_r(K_n)$. However, for $1\le r<3$ they are not the same ideal. Thus for $r=1,2$, we can calculate the Betti numbers $\beta_{i,i+r}(R/I_r(P_n^c))$ using \Cref{alternate formula} and obtain the following two combinatorial identities.
     \begin{align*}
         i\binom{n-1}{i+1}&=n\binom{n-1}{i}-\binom{n}{i+1}-(n-1)\binom{n-2}{i-1}\\
         \binom{i+1}{2}\binom{n}{i+2}-i\binom{n-2}{i}&=\binom{n}{i+2}-n\binom{n-1}{i+1}+\binom{n}{2}\binom{n-2}{i}-(n-2)\binom{n-3}{i-1}.
     \end{align*}
 \end{remark}


\section{Fr\"oberg's theorem via collapsibility}\label{collapsibility}

In this section we use the concept of $d$-collapsibility from topological combinatorics to give an alternative proof of \Cref{maintheorem}. The notion of $d$-collapsibility was introduced by Wegner in \cite{wegner}.

\begin{definition}
Let $\Delta$ be a simplicial complex. A face $\sigma \in \Delta$ is called {\it d-collapsible} if there is only one facet $\tau=\tau(\sigma)$ in $\Delta$ containing $\sigma$ (possibly $\tau=\sigma$), and moreover dim$(\sigma)\le d-1$. In this case $(\sigma, \tau)$ is called a {\it free pair} and we say that the complex $\Delta$ {\it elementary d-collapses} onto the subcomplex $\Delta'=\Delta\setminus \{\gamma\in \Delta : \sigma \subseteq \gamma\}$. We denote this collapsing by $\Delta\searrow \Delta'$. Complex $\Delta$ is called {\it d-collapsible} if there is a sequence $$\Delta \searrow \Delta_1 \searrow \Delta_2\searrow \dots \searrow \emptyset,$$
of elementary d-collapses ending with the empty complex. Note that if a complex $\Delta$ is $d$-collapsible for some integer $d\geq 1$, then it is $r$-collapsible for any integer $r\geq d$.
\end{definition}

A simplicial complex $\Delta$ is called {\it $d$-Leray} if $\tilde{H}_i(Y;\mathbb{Z})=0$ for all induced subcomplexes $Y \subseteq \Delta$ and for all $i\geq d$. The {\it Leray number} of $\Delta$, denoted $L(\Delta)$, is the minimal $d$ such that $\Delta$ is $d$-Leray. From Hochster's formula \cite{Hoch77}, we know that 
\begin{equation}\label{relation_reg_and_leray}
    {\rm reg}(I_{\Delta})=L(\Delta)+1.
\end{equation}

We know that the ideal $I_r(G)$ has a linear resolution if and only if reg$(I_r(G))=r+1$ \cite[Lemma 49]{AVTBook}).  Thus, to prove \Cref{maintheorem}, it is enough to show that reg$(I_r(G))=r+1$. We do this by showing that the complex $\Indr{r}(G)$ is $r$-Leray (recall that $I_r(G)=I_{\Indr{r}(G)}$) for any $r\geq 1$ and for any co-chordal graph $G$, which along with \Cref{relation_reg_and_leray}, will prove \Cref{maintheorem}. 
{ For more discussion on the Leray number of a complex $\Delta$ and the regularity of its Stanley-Reisner ideal $I_\Delta$, the reader is referred to \cite{KM}.}

Wegner \cite{wegner} proved that every $d$-collapsible complex is $d$-Leray. However, the converse is not true (see \cite{MT} for examples). Hence, the following result along with the discussion above gives an alternate proof of \Cref{maintheorem}.

\begin{theorem}\label{theorem:chordal collapsing}
Let $G$ be a finite simple graph and let $r\geq 1$ be an integer. If $G$ is the complement of a chordal graph, then ${\Indr{r}(G)}$ is $r$-collapsible.
{In particular, 
$I_r(G)$ has a $(r+1)$-linear resolution.}
\end{theorem}

 \begin{proof}
It is known \cite{wegner} that the complex $\Indr{1}(G)$ is $1$-collapsible for any co-chordal graph $G$ and hence $\Indr{1}(G)$ is $r$-collapsible for any $r\ge 1$. Therefore, we may assume that $r\ge 2$. 

Let $\Delta=\Indr{r}(G)$ and $F$ be a face of $\Delta$. By definition, each connected component of the induced subgraph $G[F]$ has at most $r$ vertices. Let $\mathcal C_1$ and $\mathcal C_2$ be two connected components of $G[F]$. We first show that either $|V(\mathcal C_1)|=1$ or $|V(\mathcal C_2)|=1$. On the contrary, let $v_1,v_2\in V(\mathcal C_1)$ and $v_3,v_4\in V(\mathcal C_2)$ such that $\{v_1,v_2\}$ and $\{v_3,v_4\}$ are edges in the graph $G$. In that case, $v_1v_4v_2v_3v_1$ becomes an induced cycle of length four in the chordal graph $G^c$, which is a contradiction. Hence, if $F$ is a face of $\Delta$ such that $|F|\ge r$, then $F=F'\sqcup F''$, where both $F'$ and $F''$ are faces of $\Delta$ such that $|F'|\le r$, the induced subgraph $G[F']$ is connected and the induced subgraph $G[F'']$ consists of only isolated vertices. Moreover, no vertex in $F'$ is connected to any vertex in $F''$ by an edge. Now we proceed to show that $\Delta$ is $r$-collapsible.

It is easy to see that $\Indr{1}(G)\subseteq \Delta$. Here, we show that the complex $\Delta$ can be collapsed onto $\Indr{1}(G)$ using elementary $r$-collapses. We know that if $F$ is a face of $\Delta$ such that $|F|\ge r$, then the induced subgraph $G[F]$ has at most one connected component of cardinality more than one (and at most $r$). To collapse $\Delta$ onto $\Indr{1}(G)$, it is enough to collapse all the faces in which the cardinality of the connected component is more than one. We do these collapses in the decreasing order of cardinality of the connected components, that is, we first collapse those faces in which the connected component has vertex cardinality $r$. 

Let $F$ be a face of $\Delta$ such that $G[F]$ is connected and $|F|=r>1$. 
To show that $F$ is $r$-collapsible, it is enough to show that $F$ is contained in a unique facet. 
On the contrary, let $F_1$ and $F_2$ be two facets of $\Delta$ such that $F\subseteq F_1 \cap F_2$. 
In that case if $F_1'=F_1\setminus F$ and $F_2'=F_2\setminus F$ then the induced subgraphs $G[F_1']$ and $G[F_2']$ consists of only isolated vertices. 
Since $F_1$ and $F_2$ are different facets, $F_2'\setminus F_1' \neq \emptyset$ and $F_1'\setminus F_2' \neq \emptyset$. If there is no edge in the induced subgraph $G[F_1'\cup F_2']$, then $F_1\cup F_2$ is a face of $\Delta$ which contradicts the fact that $F_1$ and $F_2$ are facets. 
Therefore, let $v \in F_2'\setminus F_1'$ such that $v$ is connected to some $w\in F_1'$. Then $vu_1wu_2v$ forms an induced cycle of length $4$ in the graph $G^c$, where $u_1,u_2\in F$ such that $\{u_1,u_2\}\in E(G)$. 
{Note that such an edge exists since $G[F]$ is a connected graph on at least two vertices.} This gives us a contradiction since $G^c$ is assumed to be chordal. Thus if $F\in \Delta$ such that $|F|=r$ and $G[F]$ is connected, then $F$ is contained in a unique facet. Moreover, note that if $F,F'$ are two different faces in $\Delta$ with $|F|=|F'|=r$ and the induced subgraphs $G[F]$ and $G[F']$ are both connected such that $F_1$ and $F_1'$, are the unique facets in $\Delta$ containing $F$ and $F'$, respectively, then $F_1\neq F_1'$. Indeed, if $F_1=F_1'$ then either $G[F\cup F']$ is contained in a connected component $\mathcal C$ of $G[F_1]$, or there exists two connected components $\mathcal C_1$ and $\mathcal C_2$ of $G[F_1]$ such that $\mathcal C_1$ and $\mathcal C_2$ contains $G[F]$ and $G[F']$, respectively. In the first case we have a contradiction since $|F\cup F'|>r$. In the second case, $G[F_1]$ has at least two connected components of vertex cardinality more than one, again a contradiction.

After collapsing all the faces $F\in \Delta$ such that $G[F]$ is connected and $|F|=r$, we do the same with faces $F'\in \Delta$ such that $|F'|=r-1>1$ and $G[F']$ is connected. Continuing this way, let $\Delta_i$ be the simplicial complex obtained from $\Delta$ by collapsing all $F\in\Delta$ such that $|F|>i$ and $G[F]$ are connected. Note that if $\sigma\in\Delta_i$, then $\sigma$ is also a face of $\Delta$ and hence $\sigma=\sigma'\sqcup\sigma''$, where both $\sigma'$ and $\sigma''$ are faces of $\Delta_i$, $|\sigma'|\le i$ where the induced subgraph $G[\sigma']$ is connected and the induced subgraph $G[\sigma'']$ consists of only isolated vertices. Moreover, no vertex in $\sigma'$ is connected to any vertex in $\sigma''$ by an edge. Proceeding as above we see that if $\sigma$ is a face of $\Delta_i$ such that $G[\sigma]$ is connected and $|\sigma|=i>1$, then $\sigma$ is contained in a unique facet of $\Delta_i$. Moreover, if $\sigma,\sigma'\in\Delta_i$ with $|\sigma|=|\sigma'|=i$ and the induced subgraphs $G[\sigma]$ and $G[\sigma']$ are both connected such that $\sigma_1$ and $\sigma_1'$ are the unique facets containing $\sigma$ and $\sigma'$, respectively, then $\sigma_1\neq\sigma_1'$. Thus each face of $\Delta_i$ of cardinality $i$ whose corresponding induced subgraph is connected, is $r$-collapsible. We continue these collapses till we collapse every face $\sigma\in \Delta$ such that $|\sigma|>1$ and $G[\sigma]$ are connected. These collapses are done in the decreasing order of the cardinality of $\sigma$. Observe that the remaining complex is $\Indr{1}(G)$ which is $1$-collapsible. This completes the proof of \Cref{theorem:chordal collapsing}.
\end{proof}

The $r$-independence complexes ($r\geq 1$) of graphs we encountered in \Cref{gap free corollary} are $r$-collapsible as we now show.

\begin{corollary}
    Let $G$ be a gap-free and claw-free graph such that it contains a leaf. Then $\mathrm{Ind}_r(G)$ and $\mathrm{Ind}_r(G^c)$ are $r$-collapsible for any integer $r\ge 1$.
\end{corollary}

\begin{proof}
    Proceeding as in Corollary \ref{gap free corollary}, we have both $G$ and $G^c$ chordal. Therefore, by the proof of Theorem \ref{theorem:chordal collapsing}, both $\mathrm{Ind}_r(G)$ and $\mathrm{Ind}_r(G^c)$ are $r$-collapsible.
\end{proof}

We know that if $G=C_n^c$, where $n\ge 4$, then $I_{\Indr{1}(G)}$ does not have a linear resolution by Fr\"oberg's theorem, since $G^c = C_n$ is not a chordal
graph for $n \geq 4$. However, in what follows we show that $I_{\Indr{r}(C_n^c)}$ has a linear resolution for each $r\ge 2$. 
{Consequently, the converse
of Theorem \ref{maintheorem} cannot hold.}
We prove that $I_{\Indr{r}(G)}$ 
has a linear resolution as a corollary of the following result.

\begin{proposition}\label{converseCondition}
Let $G$ be a graph so that $G^c$ does not have an induced cycle of length $4$. If $\mathrm{dim}\,\Indr{1}(G)\le r-1$, then $I_{\Indr{r}(G)}$ has a $(r+1)$-linear resolution. 
\end{proposition}

\begin{proof}
Following the proof of Theorem \ref{theorem:chordal collapsing} we see that if $G^c$ does not have an induced $4$-cycle then for each $r\ge 2$, $\Indr{r}(G)$ is $r$-collapsible to $\Indr{1}(G)$. Since $\mathrm{dim}\,\Indr{1}(G)\le r- 1$ we see that $\Indr{r}(G)$ is $r$-collapsible to the empty complex. Therefore, $\Indr{r}(G)$ is $r$-Leray. Consequently, $I_{\Indr{r}(G)}$ has a linear resolution.
\end{proof}

\begin{corollary}\label{cor.conversefalse}
For $n\geq 4$, the ideal $I_{\Indr{r}(C_n^c)}$ has a linear resolution for each $r\ge 2$.
\end{corollary}
\begin{proof}
For $n\le 4$, $I_{\Indr{r}(C_n^c)}$ is the zero ideal when $r\ge 2$ and hence has a linear resolution. For $n\ge 5$, note that $\mathrm{dim}\,\Indr{1}(C_n^c)= 1$ and hence the result follows from Proposition \ref{converseCondition}.
\end{proof}

\section{Concluding Remarks}\label{conrem}
As stated in the Introduction, determining a combinatorial description of higher degree square-free monomial ideals that have a linear resolution over all fields is an active area of research. 
A prominent setting to undertake such a study is that of hypergraphs, since they generalize graphs and the edge ideal of a hypergraph is a square-free monomial ideal. 
Inspired by Fr\"oberg's theorem, there were attempts to generalize the notion of \emph{chordality} to hypergraphs and then prove that the edge ideal of the complement hypergraph has a linear resolution. 
A recent approach is by Bigdeli, Yazdan Pour and Zaare-Nahandi \cite{BYZ}. 
The authors first introduce the notion of chordal hypergraphs (by generalizing the perfect elimination order of chordal graphs) and then show that the edge ideal associated with the complement hypergraph has a linear resolution over any field. 
They further show that this particular class of chordal hypergraphs contains several, previously defined classes of chordal hypergraphs. 
More recently, Bigdeli and Faridi \cite{BF} have extended this notion of chordality to the realm of simplicial complexes. 

Let $\mathcal{H}$ be an $(r+1)$-uniform hypergraph (i.e., every edge is of cardinality $r+1$) on the vertex set $V$.
A subset $W\subset V$ is called a \emph{simplical maximal subcircuit} if $|W| = r$, it is contained in an edge, and its neighborhood is a clique (see \cite[Definition 1.2]{BYZ} for details). 
Now a \emph{chordal hypergraph} is recursively defined as the $(r+1)$-uniform hypergraph $\mathcal{H}$ which is either empty or contains a simplicial maximal subcircuit $W$ such that the deletion $\mathcal{H}\setminus W$ is also chordal. 
Further, a hypergraph is \emph{co-chordal} if its complement is chordal. 
We have already seen in the Introduction that the notion of $r$-independence lets us define a hypergraph, which we denote by $\mathrm{Con}_r(G)$. 
Based on our calculations we propose the following conjecture: 

\begin{conjecture}
For all $r\geq 1$, $\mathrm{Con}_r(G)$ is a co-chordal hypergraph if $G$ is a co-chordal graph.
\end{conjecture}
The reader should note that our work allows us to directly show that $I_r(G)$ has a linear resolution without first checking if ${\rm Con}_r(G)$ is chordal or not. 
Moving on, because of \Cref{cor.conversefalse} the following question is worth answering:

\begin{question}
Are there examples of (non-co-chordal) graphs $G$ such that $I_r(G)$  has a linear resolution but the hypergraph $\mathrm{Con}_r(G)$ is not a co-chordal hypergraph?
\end{question}


\end{document}